\documentclass[a4paper]{article}
\usepackage{amsmath, amsthm, amssymb, pdfsync, psfrag, verbatim,color}
\usepackage{hyperref}
\usepackage {mathrsfs, graphicx, newcent,wrapfig,multirow}
\usepackage{breakcites}
\usepackage[round]{natbib}
\usepackage[affil-it]{authblk}
\usepackage{bbm}
\newcommand{\bbD}{\mathbb D}
\newcommand{\bbE}{\mathbb E}

\newcommand{\bbN}{\mathbb N}

\newcommand{\bbP}{\mathbb P}

\newcommand{\bbR}{\mathbb R}

\newcommand{\scB}{\mathcal B}
\newcommand{\scC}{\mathcal C}
\newcommand{\scD}{\mathcal D}

\newcommand{\scF}{\mathcal F}
\newcommand{\scG}{\mathcal G}

\newcommand{\scL}{\mathcal L}

\newcommand{\veps}{\varepsilon}

\newcommand{\ang}[1]{\ensuremath{ \left \langle #1 \right \rangle }}

\newcommand{\abs}[1]{\ensuremath{\left| #1 \right|}}

\DeclareMathOperator*{\esssup}{ess\,sup}

\newcommand{\indicator}[1]{\ensuremath{\mathbf{1}_{\crl{#1}}}}

\newcommand{\crl}[1]{\ensuremath{ \left\{ #1 \right\} }}
\newcommand{\edg}[1]{\ensuremath{ \left[ #1 \right] }}
\newcommand{\brak}[1]{\ensuremath{\left( #1 \right)}}

\newtheorem{theorem}{Theorem}[section]
\newtheorem{definition}[theorem]{Definition}

\newtheorem{lemma}[theorem]{Lemma}
\newtheorem{remark}[theorem]{Remark}
\newtheorem{example}[theorem]{Example}
\newtheorem{examples}[theorem]{Examples}
\newtheorem{foo}[theorem]{Remarks}
\newtheorem{assumption}[theorem]{Assumption}
\newenvironment{Example}{\begin{example}\rm}{\end{example}}

\newenvironment{Remark}{\begin{remark}\rm}{\end{remark}}

\numberwithin{equation}{section}
\title{No arbitrage assumption implies the differentiability of derivative pricing function}
\author{Kihun Nam}
\author{Yunxi Xu}
\affil{Department of Mathematics, Monash University}
\date{\today}

\begin{document}	
\maketitle
 	\begin{abstract}
The no arbitrage assumption implies that the price of an asset must be a semimartingale. In this article, we characterize the class of functions that map It\^o processes to continuous semimartingales, as well as those that map continuous Markov semimartingales to continuous Markov semimartingales. This class of functions generalizes the conventional Sobolev space by adopting a weaker notion of derivatives. In particular, the functions are required to be weakly differentiable with respect to the input process. From a financial perspective, our results show that the no arbitrage assumption implies the differentiability of any derivative price with respect to the underlying asset, provided that the underlying noise is a continuous Markov semimartingale. Moreover, we demonstrate that the Malliavin differentiability of the input process (the underlying) leads to the Malliavin differentiability of the output process (the derivative price).
		\\[2mm]
	\end{abstract}
\section{Introduction}
We consider a financial market on the filtered probability space where the filtration is generated and augmented by a Brownian motion. Assume there is ``No Free Lunch with Vanishing Risk'' (NFLVR). \cite{delbaen1994general, delbaen1998fundamental} showed NFLVR is equivalent to the existence of an equivalent sigma martingale measure. Then, the Girsanov theorem and the Bichteler-Dellacherie theorem tells us every price process should be a semimartingale.

On the other hand, consider a process $Y_t=f(t,X_t)$ where $f:\bbR\times\bbR\to\bbR$ and $X$ is a right Markov process. When $f$ is a time-homogeneous, \cite{cinlar1980semimartingales} proved $Y$ is a semimartingale if and only if $f$ is a difference between excessive functions. In particular, when the Markov process is a Brownian motion, then they showed the resulting process is a semimartingale if and only if the function is the difference of two convex functions. The result is generalised to the case where the function is time-inhomogeneous by \cite{chitashvili1997functions}, and they chracterised the space for $f$. When the $X$ is an It\^o process, the same authors obtained the necessary and sufficient condition for the time-inhomogeneous $f$ to make $Y$ an It\^o process in \cite{chitashvili1996generalized}. On the other hand, the It\^o formula tells us that, if $X$ is a semimartingale and $f$ is in $\scC^{1,2}$, then $Y_t=f(t,X_t)$ is a semimartingale. When $X$ is a solution of an SDE driven by a continuous semimartingale, we characterise the class of $f$ that makes $f(t,X_t)$ a continuous semimartingale.

These two streams of idea naturally lead to our main claim: when $X$ is a continuous Markov semimartingale, the no-arbitrage principle (NFLVR) on $Y_t = f(t, X_t)$ implies a form of regularity on the function $f$ in the transformation. In particular, we show that $f$ must possess a certain degree of differentiability. For example, a direct consequence is that $f(t,x)$ should be weakly differentiable with respect to $x$, which ensures the existence of a well-defined \textit{delta} for European-style derivatives written on the asset $X$. This aligns with our intuition since delta usually represents our hedging of the derivative. On the contrary, it is possible that \textit{theta} and 
\textit{gamma} may not be well-defined contrasting to the conventional Sobolev-differentiability assumptions. Another implication is that the differentiability of $X$ with respect to the underlying noise transfers to the differentiability of $Y$ via the chain rule: e.g. $Y$ should be Malliavin differentiable if $X$ is Malliavin differentiable. These insights not only extend the classical results of \cite{cinlar1980semimartingales} and \cite{chitashvili1996generalized,chitashvili1997functions} but also bridge the gap between stochastic analysis and financial modeling by characterizing the minimal regularity needed for semimartingale-based asset pricing.

The space of semimartingale functions is known to be strictly weaker than the classical Sobolev space. Due to its intrinsic connection with stochastic calculus and broad applicability in financial modeling, particularly in the analysis of asset price dynamics, it becomes essential to investigate the structure and properties of the function space that transforms a Markov process into a semimartingale. Understanding this transformation not only deepens our theoretical grasp of stochastic processes but also provides valuable tools for the development of robust models in quantitative finance, optimal control, and related fields. 


\section{Functions transforming an It\^{o} process into semimartingale}
Let $\brak{\Omega,\scF,(\scF_t)_{t\geq 0},\bbP}$ be a filtered probability space with the filtration $(\scF_t)_{t\geq 0}$ satisfying the usual conditions. We use $|\cdot|$ for the Euclidean norm on Euclidean spaces. We denote $\bbD^{1,2}$ as the Malliavin Sobolev space of Malliavin differentiable random variables with square integrable derivatives. We consider an It\^{o} process $X$ such that it satisfies a SDE of the following form
\begin{align}
    \label{ito sde}
    dX_t=b(t,X_t)dt+\sigma(t,X_t)dW_t;\quad X_0=x\in\bbR^m
\end{align}
where $W$ is an adapted standard $\bbR^m$-valued Brownian motion and $b,\sigma$ are functions from $\bbR^+\times\bbR^m$ with values in $\bbR^m$ and $\bbR^{m\times m}$, respectively.
We assume the following:
\begin{assumption}
\label{sde assumption}
    \begin{itemize}
        \item[(A1)] The coefficient $b$ is measurable and bounded.
        \item[(A2)] The coefficient $\sigma\sigma^\intercal$ is bounded and continuous. Moreover, there exists a constant $\veps>0$ such that
			\[
			\veps|x'|^2\leq (x')^\intercal(\sigma\sigma^\intercal)(t,x) x'\leq \veps^{-1}|x'|^2\\
			\]
			hold for all $t\in[0,T], x,x'\in\bbR^m$.
\end{itemize}
\end{assumption}

Notice that, under Assumption \ref{sde assumption}, it is well-known there exits a unique weak solution of SDE \eqref{ito sde} and it satisfies the strong Markov property.

Next, let us remind the operator $L$ and the corresponding domain $V^L_\mu(loc)$ and $\hat V^L_\mu(loc)$, which are introduced in \cite{chitashvili1996generalized} and \cite{chitashvili1997functions}. Let $\mu(ds,dy):=p(s,y)dsdy$, where $p$ is the transition density corresponding to SDE $dZ_t=\sigma(t,Z_t)dW_t$ with $Z_0=x\in\bbR^m$.
For a function $f\in C^{1,2}$, the $L$ operater is defined by
\begin{align*}
    \brak{Lf}(t,x)=f_t(t,x)+\frac{1}{2} \sum_{i,j=1}^m(\sigma\sigma^{\intercal})_{ij}(t,x)f_{x_ix_j}(t,x),
\end{align*}
where $f_t,f_{x_ix_j}$ are partial derivatives of the function $f$. 
\begin{definition}
\label{def V}

		We say $f$ belongs to $V^{L}_\mu(loc)$, if there exists a sequence of functions $(f_n)_{n\geq 1}\subset C^{1,2}$, and a measurable function $L f$ which satisfy the following conditions:
		\begin{itemize}
		\item[(i)] For every compact set $D\in \bbR^+\times\bbR^m$,
		\begin{align*}
						\sup_{(s,x)\in D}|f^n(s,x)-f(s,x)|\xrightarrow{n\to\infty}0.
		        \end{align*}
		\item[(ii)] There exists a sequence of functions $(h_k(t,k))_{k\geq 1}$ with the following properties: 
		\begin{itemize}
		\item $h_k(0,x)=1$ for each $x\in\bbR^m$, $h_k(t,x)\leq h_{k+1}(t,x)$, $h_k(t,x)\rightarrow 1$ $\mu$-a.e., and $\tau_k:=\inf\{t:h_k(t,X_t)\leq\lambda\}$, for some $0<\lambda<1$, are stopping times with $\tau_k\rightarrow\infty$ $\bbP$-a.s.
		\item For each $k\geq 1$,
							\begin{align*}
								\iint|L f^n(s,x)-L f(s,x)|h_k(s,x)\mu(ds,dx)\xrightarrow{n\to\infty}0.
							\end{align*}
		\end{itemize} 
		\end{itemize}
Then, we define the $L$-derivative of $f$ by $L f$. 	Moreover, if $f\in V^{\scL}_\mu(loc)$, then there exists $f_x(t,x)$ such that
		\begin{align}
        \label{fx convergence}
			\iint|f^n_x(s,x)-f_x(s,x)|^2 h_k(s,x)\mu(ds,dx)\xrightarrow{n\to\infty}0.
		\end{align}
		We define $f_x$ to be the generalized gradient of $f$.
\end{definition}
\begin{Remark}
\label{equivalent def}
    Notice that the above definition is equivalent to the following:
    A function $f$ belongs to $V^{L}_\mu(loc)$, if there exists a sequence of functions $(f_n)_{n\geq 1}\subset C^{1,2}$, a sequence of bounded measurable domains $D_1\subset D_2\subset\cdots$ with $(0,x)\in D_1$ and $\cup_{n\in\bbN} D_n=\bbR^+\times\bbR^m$, and a measurable locally $\mu$-integrable function $Lf$ such that for $u_k:=\inf\{t:(t,X_t)\notin D_k\}$ and each $k\geq 1$,
    \begin{align*}
      &\sup_{s\leq u_k}\abs{f^n(s,X_s)-f(s,X_s)}\xrightarrow{n\to\infty}\infty\\
     \iint_{D_k}&|L f^n(s,x)-L f(s,x)|\mu(ds,dx)\xrightarrow{n\to\infty}0.
    \end{align*}
This is a direct consequence of Lemma $5$ and Lemma $6$ in \cite{chitashvili1996generalized} plus Proposition $4$ in \cite{chitashvili1997functions}.
\end{Remark}
\begin{definition}
    \label{hat V}
    We say a function $f$ belongs to $\hat{V}^L_\mu(loc)$ if there exists a sequence of functions $(f^n)_n$ from $C^{1,2}$ and a $\sigma$-finite signed measure $\nu^L_f$ on $(\bbR^+\times\bbR^m,\scB(\bbR^+\times\bbR^m))$ which satisfy the following conditions:
    \begin{itemize}
    \item[(i)]  The condition (i) from Definition \ref{def V} holds.
    \item[(ii)] There exists a sequence of functions $(h_k(t,k))_{k\geq 1}$ with the following properties: 
    		\begin{itemize}
    		\item $h_k(0,x)=1$ for each $x\in\bbR^m$, $h_k(t,x)\leq h_{k+1}(t,x)$, $h_k(t,x)\rightarrow 1$ $\mu$-a.e., and $\tau_k:=\inf\{t:h_k(t,X_t)\leq\lambda\}$, for some $0<\lambda<1$, are stopping times with $\tau_k\rightarrow\infty$ $\bbP$-a.s.
    		\item For each $k\geq 1$,
    							\begin{align*}
    							            \iint \psi(s,x)\brak{Lf^n}(s,x)h_k(s,x)\mu(ds,dx)\xrightarrow{n\to\infty}\iint \psi(s,x)h_k(s,x)\nu^L_f(ds,dx)
    							        \end{align*}
    							        for every bounded continuous function $\psi$.
    		\end{itemize} 
    \end{itemize}
\end{definition}
Now we are ready to state our first main theorem.
\begin{theorem}
\label{decomposition theorem}
  Suppose $X$ satisfies Assumption \ref{sde assumption}.  The process $(f(t,X_t))_{t\geq 0}$ is an It\^{o} process if and only if $f\in V^{L}_\mu(loc)$ and it admits the decomposition
\begin{align*}
    f(t,X_t)=f(0,X_0)+\int_0^t f_x(s,X_s)dX_s+\int_0^tLf(s,X_s)ds.
\end{align*}
On the other hand, the process $(f(t,X_t))_{t\geq 0}$ is a semimartingale if and only if $f\in \hat{V}^L_\mu(loc)$, and it has decomposition
\begin{align}
\label{hat V representation}
    f(t,X_t)=f(0,X_0)+\int_0^t f_x(s,X_s)dX_s+ A_t^f
\end{align}
where $A^f$ is uniquely determined by the relation
\begin{align}
    \bbE\int_0^\infty\psi(s,X_s)dA_s^f=\iint\psi(s,x)\nu^L_f(ds,dx)
\end{align}
for each bounded continuous function $\psi$.
\end{theorem}
\begin{proof}
    The first claim is proved in \cite{chitashvili1996generalized}. Since the reference is heavily used in this proof, we will use CM to refer to this paper. Without loss of generality, let us prove the second claim for the case $b=0$. The general case $b\neq 0$ follows immediately from Girsanov's Theorem. The idea is essentially the same to the proof of Theorem 1 of \cite{chitashvili1997functions} with modification for $X$ being an It\^o process instead of a Brownian motion.\\
\textbf{Step 1.} ($f\in\hat{V}^{\scL}_\mu(loc)\implies f(\cdot,X.)$: semimartingale)  Let $(h_k(s,x))_{k\geq 1}$ be a localizing sequence of functions from Definition \ref{hat V}, and $(\tau_k)_{k\geq 1}$ be corresponding stopping times with $\tau_k=\inf \{t:h_k(t,X_t)\leq\lambda<1\}$. Since $\tau_k\rightarrow\infty$, it is enough to show $(f(t\land\tau_k,X_{t\land\tau_k}))_{t\geq 0}$ is a semimartingale for every $k\geq 1$. For notation simplicity, we omit $k$ in the following proof. 
    
    Let $(f^n)_n$ be the approximating sequence of $C^{1,2}$ functions from Definition \ref{hat V}. Thus $f^n(t\land\tau,X_{t\land\tau})$ is a semimartingale with decomposition
    \begin{align*}
        f^n(t\land\tau,X_{t\land\tau})=f^n(0,X_0)+M^n_t+A^n_t,
    \end{align*}
where \begin{align*}
    M^n_t:=\int_0^{t\land \tau}f^n_x(s,X_s)\sigma(s,X_s)dW_s\qquad \text{and}\qquad A^n_t:=\int_0^{t\land \tau} \brak{Lf^n}(s,X_s)ds.
\end{align*}
It follows from (i) that for each $t\geq 0$,
\begin{align*}
    \sup_{s\leq t}\abs{f^n(s\land\tau,X_{s\land\tau})-f(s\land\tau,X_{s\land\tau})}\xrightarrow{n\rightarrow\infty}0\  a.s.
\end{align*}
According to Proposition 3 in \cite{chitashvili1997functions}, it is sufficient to show that, for each $t>0$, 
\begin{align*}
    \lim_N\sup_n\bbP\brak{\brak{Var A^n}_{t\land\tau}>N}=0.
\end{align*}
where $Var A^n$ is the total variation of $A^n$.
Notice that
\begin{align*}
    \bbE\brak{Var A^n}_{t\land\tau}&=\bbE\int_0^{t\land\tau}\abs{\brak{Lf^n}(s,X_s)}ds\leq \frac1\lambda \bbE\int_0^{t\land\tau}h_k(s,X_s)\abs{\brak{Lf^n}(s,X_s)}ds\\
    &\leq \frac1\lambda\iint h_k(s,x)\abs{\brak{Lf^n}(s,x)}\mu(ds,dx).
\end{align*}
Moreover, we have 
\begin{align*}
        \iint \psi(s,x)\brak{Lf^n}(s,x)h_k(s,x)\mu(ds,dx)\xrightarrow{n\to\infty}\iint \psi(s,x)h_k(s,x)\nu^L_f(ds,dx)
    \end{align*}
    for every bounded continuous function $\psi$. Thus, for each $k$,
    \begin{align*}
        \sup_n\iint h_k(s,x)\abs{\brak{Lf^n}(s,x)}\mu(ds,dx)<\infty.
    \end{align*}
Thus, $(f(t\land\tau_k,X_{t\land\tau_k}))_{t\geq 0}$ is a semimartingale for every $k\geq 1$. In other words, $f(\cdot,X.)$ is a semimartingale.\\
Now we proceed to prove $f$ has the generalized gradient in $x$. Without loss of generality, we consider the case where $Z_t:=f(t,X_{t})=Z_0+M_t+A_t$ is the limit of $Z^n_t:=f^n(t,X_{t})=Z^n_0+M^n_t+A^n_t$. Here, $M,M^n$ are local martingales and $A,A^n$ are predictable finite variation process. From Proposition 3 in \cite{chitashvili1997functions}, we know there exists a subsequence of sequence $(f^n)_n$ which we still denote by $(f^n)_n$ and a sequence of stopping times $u_k$ with $u_k\xrightarrow{k\rightarrow\infty}\infty$ such that for every $k$,
\begin{align*}
    \bbE\langle M^n-M \rangle_{u_k}\rightarrow 0\ \text{and}\ \bbE\sup_{s\leq u_k}\abs{A_s^n-A_s}\rightarrow 0.
    \end{align*}
Consider another localizing functions $\Tilde{h}_k(s,x)=\bbE\edg{\indicator{s<u_k}|X_s=x}$ and $C_k=\{(s,x):\Tilde{h}_k(s,x)>\lambda\}$ for some $0<\lambda<1$. According to the proof of Lemma 6 in CM, 
\begin{align*}
    \Tilde{\tau}_k=\inf\{t:\Tilde{h}_k(t,X_t)\leq\lambda\}\wedge u_k=\inf\{t:(t,X_t)\notin C_k\}\wedge u_k,
\end{align*}
are stopping times, and thus 
\begin{align}
\label{general deriviative inequality}
\begin{split}
    \iint_{C_k}&\brak{f^n_x(s,x)-f^m_x(s,x)}^2\mu(ds,dx)\leq \varepsilon^{-1}\iint_{C_k} \brak{f^n_x(s,x)-f^m_x(s,x)}^2 (\sigma\sigma^\intercal)(s,x)\mu(ds,dx)\\
    &\leq \varepsilon^{-1}\bbE\int_0^{\tilde\tau_k}\brak{f^n_x(s,X_s)-f^m_x(s,X_s)}^2 (\sigma\sigma^\intercal)(s,X_s)ds\\
    &\leq \frac{1}{\lambda\varepsilon}\bbE\int_0^{\tilde\tau_k}\tilde h_k(s,X_s)\brak{f^n_x(s,X_s)-f^m_x(s,X_s)}^2 (\sigma\sigma^\intercal)(s,X_s)ds\\
    &\leq\frac{1}{\lambda\varepsilon} \bbE\int_0^{u_k}\brak{f^n_x(s,X_s)-f^m_x(s,X_s)}^2 (\sigma\sigma^\intercal)(s,X_s)ds=\frac{1}{\lambda\varepsilon} \bbE\langle M^n-M^m\rangle_{u_k}\xrightarrow{n,m\rightarrow\infty}0.
\end{split}
\end{align}
Thus, from the completeness of $L^2$ space, there exists $f_x$ such that
\begin{align*}
    \iint_{C_k} \abs{f^n_x(s,x)-f_x(s,x)}^2\mu(ds,dx)\xrightarrow{n\rightarrow\infty}0.
\end{align*}
This indicates $f$ has locally generalized gradient in $x$ w.r.t. the measure $\mu$. Note we can also have 
\begin{align*}
    \bbE\left\langle\int_0^\cdot f^n_x(s,X_s)dX_s-M\right\rangle_{u_k}\xrightarrow{n\rightarrow\infty} 0.
\end{align*}
Since $u_k\rightarrow\infty$, we obtain $M_t=\int_0^t f_x(s,X_s)dX_s$ a.s. for all $t$.\\
\textbf{Step 2.} ($f(\cdot,X.)$: semimartingale $\implies f\in\hat{V}^{\scL}_\mu(loc)$): Let $(f(t,X_t))_t$ be a continuous semimartingale and $Z_t=f(t,X_t)=Z_0+A_t+M_t$. First, let us assume $f(t,X_t)$ is bounded for all $t$, that is $\sup_t\abs{f(t,X_t)}\leq C$ for some constant $C$. Furthermore, let $(\tau_k)_k$ be a nondecreasing sequence of stopping times with $\tau_k\xrightarrow{k}\infty$ such that $\bbE(Var A)_{\tau_k}<\infty$ and $\bbE\langle M,M\rangle_{\tau_k}<\infty$ for every $k\geq 1$. Consider a sequence of functions $(\hat{f}^n)_n$
\begin{align*}
    \hat{f}^n(s,x)=n\int_s^{s+1/n}\int_{\bbR^m}f(u,y)P(s,x,u,dy)du,
\end{align*}
where $P$ stands for transition density function of $X$.
According to Lemma 4 in CM we know that for each $n$, $\hat{f}^n\in V_\mu^L(loc)\subset \hat{V}_\mu^L(loc)$ and
\begin{align*}
    \brak{L\hat{f}^n}(s,x)=n\int_{\bbR^m}\brak{f(s+1/n,y)-f(s,x)}P(s,x,s+1/n,dy).
\end{align*} 
By Theorem 1 in CM, it follows that $\hat{f}^n(t,X_t)$ is a semimartingale such that 
\begin{align*}
    \hat{f}^n(t,X_t)=\hat{f}^n(0,X_0)+\Tilde{A}^n_t+\Tilde{M}^n_t
\end{align*}
where $\Tilde{M}^n_t=\int_0^t\hat{f}^n_x(s,X_s)dX_s$ and
\begin{align*}
    \Tilde{A}^n_t=\int_0^t\brak{L\hat{f}^n}(s,X_s)ds=n\int_0^t \bbE\edg{f(s+1/n,X_{s+1/n})-f(s,X_s)|X_s}ds.
\end{align*}
First, let us show that $f$ has the generalized gradient in $x$ using the same argument in the previous step based on Proposition 3 in \cite{chitashvili1997functions}. Since $\sup_{s\leq t}\abs{\hat{f}^n(s,X_s)-f(s,X_s)}\rightarrow 0$ a.s. for every $t>0$ by Lemma 1 of CM, it is sufficient to show that for each $k$, $\sup_n \bbE\brak{Var\Tilde{A}^n}_{\tau_k}<\infty$ to apply the proposition. Recall that $\sup_t\abs{f(t,X_t)}<C$ and $\bbE\edg{f(t,X_t)-f(s,X_s)|X_s}=\bbE\edg{A_t-A_s|\scF_s^X}$ on $s<t\leq\tau_k$. Now we can deduce that
\begin{align*}
    \brak{Var\Tilde{A}^n}_{\tau_k}=&n\int_0^{\tau_k}\abs{\bbE\edg{f(s+1/n,X_{s+1/n})-f(s,X_s)|X_s}}ds\\
    =&n\int_{\tau_k}^{\tau_k+1/n}\abs{\bbE\edg{f(u,X_u)-f(u-1/n,X_{u-1/n})|X_{u-1/n}}}du\\
    &+n\int_{1/n}^{\tau_k}\abs{\bbE\edg{f(u,X_u)-f(u-1/n,X_{u-1/n})|X_{u-1/n}}}du\\
    \leq&2C+n\int_{1/n}^{\tau_k}\abs{\bbE\edg{A_u-A_{u-1/n}|\scF^X_{u-1/n}}}du.
\end{align*}
Thus we have 
\begin{align}
\label{finite variation}
    \bbE\brak{Var\Tilde{A}^n}_{\tau_k}\leq 2C+n\bbE\int_{1/n}^{\tau_k}\abs{{A_u-A_{u-1/n}}}du\leq 2C+n\bbE\int_{1/n}^{\tau_k}\int_{u-1/n}^u\abs{dA_s}du.
\end{align}
By Fubini's theorem and $(s+1/n)\wedge\tau_k-s\vee 1/n\leq 1/n$, we have
\begin{align*}
    \bbE\int_{1/n}^{\tau_k}\int_{u-1/n}^u\abs{dA_s}du=\bbE\int_0^{\tau_k}\int_{s\vee 1/n}^{(s+1/n)\wedge\tau_k}du\abs{dA_s}\leq 1/n\bbE\int_0^{\tau_k}\abs{dA_s}.
\end{align*}
Thus we have $\bbE\brak{Var\Tilde{A}^n}_{\tau_k}\leq 2C+\bbE\brak{VarA}_{\tau_k}<\infty$. Therefore, there exists a subsequence of $(\hat{f}^n)_n$ and a sequence of stopping times that we still denote by $(\tau_k)_k$ such that for every $k\geq 1$,
\begin{align}
\label{tilde An convergence}
    \bbE\langle\Tilde{M}^n-M\rangle_{\tau_k}\rightarrow 0\ \text{and}\ \bbE\sup_{s\leq \tau_k}\abs{\Tilde{A}^n_s-A_s}\rightarrow 0.
\end{align}
Using the same argument in the previous step, we can conclude $f$ has generalized gradient in $x$ and thus $M_t=\int_0^tf_x(s,X_s)dX_s$. \\
Now we proceed to study the finite variation process $A_t$. Define $\nu_A(ds,dx)$ on Borel sets from $\bbR^+\times\bbR^m$: for a measurable positive function $\psi(s,x)$ with $\bbE\int_0^\infty \psi(s,X_s)\abs{dA_s}<\infty$, let $\nu_A(\psi)=\bbE\int_0^\infty\psi(s,X_s)dA_s$. Note that for the sequence of domains $D_k=\{(s,x):\bbE\edg{\indicator{s<\tau_k}|X_s=x}>\lambda\}$ for some $0<\lambda<1$, we have $\bigcup_kD_k=\bbR^+\times\bbR^m$, and from Lemma 5 in CM, we can deduce that 
\begin{align*}
    \abs{\nu_A(D_k)}\leq\bbE\int_0^\infty\indicator{D_k}(s,X_s)\abs{dA_s}\leq 1/\lambda \bbE\int_0^\infty \bbE\edg{\indicator{s<\tau_k}|X_s}\abs{dA_s}=1/\lambda \bbE\int_0^{\tau_k}\abs{dA_s}<\infty.
\end{align*}
Thus measure $\nu_A$ is $\sigma$-finite on $\scB(\bbR^+\times\bbR^m)$. Recall that for each $n$, $\hat{f}^n\in V_\mu^L(loc)$ and thus 
\begin{align*}
\hat{f}^n(t,X_t)=\hat{f}^n(0,X_0)+   \int_0^t\hat{f}^n_x(s,X_s)dX_s+\int_0^t\brak{L\hat{f}^n}(s,X_s)ds.
\end{align*}
From \eqref{finite variation} and \eqref{tilde An convergence} we have for every $k$,
\begin{align*}
    \bbE\int_0^{\tau_k}\psi(s,X_s)d\Tilde{A}^n_s\rightarrow \bbE\int_0^{\tau_k}\psi(s,X_s)dA_s
\end{align*}
for all bounded continuous function $\psi$. By Lemma 5 in CM and $\Tilde{A}^n_t=\int_0^t\brak{L\hat{f}^n}(s,X_s)ds$, we can write
\begin{align*}
    \bbE\int_0^{\tau_k}\psi(s,X_s)d\Tilde{A}^n_s=\iint\psi(s,x)\brak{L\hat{f}^n}(s,x)h_k(s,x)\mu(ds,dx),
\end{align*}
with $h_k(s,x)=\bbE\edg{\indicator{s<\tau_k}|X_s=x}$, and
\begin{align*}
    J_k^n(\psi):=\iint\psi(s,x)\brak{L\hat{f}^n}(s,x)h_k(s,x)\mu(ds,dx)\xrightarrow{n\rightarrow\infty}\bbE\int_0^{\tau_k}\psi(s,X_s)dA_s=:J_k(\psi)
\end{align*}
Thus, for every $k$, $(J_k^n)_{n=1,2,...}$ is a sequence of linear bounded functionals on bounded continuous functions $C_b$ and it converges for every $\psi\in C_b$. Therefore, its limit $J_k$ is also a linear bounded functional on $C_b$ and it is representable in the form
\begin{align*}
    J_k(\psi)=\iint\psi(s,x)\nu_k(ds,dx),
\end{align*}
where $\nu_k$ is a finite measure on $\scB(\bbR^+\times\bbR^m)$. This implies that 
\begin{align*}
\bbE\int_0^{\tau_k}\psi(s,X_s)dA_s=\iint\psi(s,x)\nu_k(ds,dx).
\end{align*}
From the definition of $\nu_A$ we have $\nu_k(ds,dx)=h_k(s,x)\nu_A(ds,dx)$ and we obtain that for each $k$
\begin{align}
\label{weak L derivative}
    \iint\psi(s,x)\brak{L\hat{f}^n}(s,x)h_k(s,x)\mu(ds,dx)\rightarrow \iint\psi(s,x)h_k(s,x)\nu_A(ds,dx).
\end{align}
Since Lemma 1 in CM implies
\begin{align}
\label{uniform convergence for weak L}
\sup_{s,x\in D}\abs{\hat{f}^n(s,x)-f(s,x)}\rightarrow 0
\end{align}
for every compact $D$ and, for any $n$, $\hat{f}^n\in V^L_\mu(loc)$, it indicates that there exists a sequence of functions from $C^{1,2}$ such that convergences \eqref{weak L derivative} and \eqref{uniform convergence for weak L} hold. Therefore, $f\in\hat{V}^L_\mu(loc)$ and we denote by $\nu^L_f:=\nu_A$ the generalized weak derivative of $f$. Note we can have
\begin{align}
    \label{weak L relation}
    \bbE\int_0^\infty\psi(s,X_s)dA_s=\iint\psi(s,x)\nu^L_f(ds,dx),
\end{align}
and combining $M_t$ we obtained previously, we have equality \eqref{hat V representation}.

\bigskip

\textbf{Step 3.} Lastly, let us show that relation \eqref{weak L relation} uniquely determines process $A$. Let $B$ be another additive process satisfies \eqref{weak L relation}. Then we have
\begin{align*}
    \bbE\int_0^{\tau_k}\psi(s,X_s)dA_s=\bbE\int_0^{\tau_k}\psi(s,X_s)dB_s.
\end{align*}
Notice that if $\psi$ is a bounded function such that $\brak{\psi(t,X_t)}_t$ is right continuous, then we can approximate by bounded continuous functions similar to $\hat f_n$. On the other hand, note that, for each adapted bounded c\`adl\`ag process $Y$, there exists a c\`adl\`ag modification of $E\edg{Y_t|\scF^X_t}$, which is a result from Theorem $6$ in \cite{rao1972modification}. Then Lemma 5 in CM, we can write for any adapted bounded c\`adl\`ag process $Y$
\begin{align*}
    \bbE\int_0^{\tau_k}Y_sdA_s=\bbE\int_0^\infty \bbE\edg{\indicator{s\leq\tau_k}Y_s|X_s}dA_s=\bbE\int_0^\infty \bbE\edg{\indicator{s\leq\tau_k}Y_s|X_s}dB_s= \bbE\int_0^{\tau_k}Y_sdB_s
\end{align*}
which means $A,B$ are indistinguishable.

To remove boundedness of $f$ we can apply a localization technique. For instance, define $\phi^C(x)=x$ if $|x|\leq C$ and $\phi^C(x)=C$ if the others. Then it is easy to see that $\phi^C(f(t,X_t))$ uniformly converges to $f(t,X_t)$ on every compact  a.s. If $(f(t,X_t))_t$ is a semimartingale then  $\brak{\phi^C(f(t,X_t))}_t$ is a bounded semimartingale for every $C$, which means $\phi^C(f)\in\hat{V}^L_\mu(loc)$, and thus $f$ belongs to the same class.
\end{proof}

Since the semimartingale property of $f(t,X_t)$  implies a regularity of $f$, the regularity of $X$ cascades to the regularity of $f(t,X_t)$. More specifically, the Malliavin differentiability of $X$ implies the Malliavin differentiability of $f(t,X_t)$ by the chain rule.
\begin{assumption}
\label{differentiable assumption}
   The coefficients $b,\sigma$ are smooth enough such that the strong/weak solution $X$ of SDE \eqref{ito sde} satisfies $X_t$ is Malliavin differentiable for almost every $t>0$ and $\esssup_{\Omega}\int_0^T|\scD_s X_t|^2 ds<\infty$.
\end{assumption}
\begin{theorem}\label{main theorem}
    Assume that $f(t,X_t)=A_t+M_t$ is a semimartingale with finite variation process $A$ and local martingale $M$ satisfying:
\begin{align}\label{sqint}
	\bbE\edg{( Var A)_T^2+\ang{M}}<\infty
\end{align}
     for any $T$, and that $X$ satisfies Assumption \ref{sde assumption} and \ref{differentiable assumption}.  Then $f(t,X_t)$ is Malliavin differentiable, that is for almost every $t$, $f(t,X_t)\in\bbD^{1,2}$.
\end{theorem}
\begin{Remark}
	Let $X$ be an asset price satisfying Assumption \ref{sde assumption} and Assumption \ref{differentiable assumption}. The previous theorem tells us that, under NFLVR assumption, any bounded financial derivatives on $X$ should be Malliavin differentiable since the Markovian property of $X$ forces the derivative price to be a measurable function of $(t,X_t)$.
\end{Remark}
\begin{Remark}
\cite{pham2013some} tells us that the condition \eqref{sqint} can be replaced by, for $Y_t:=f(t,X_t)$,
$$
\bbE\sup_{t\in[0,T]}|Y_t|^2+\sup_{\Pi}\bbE\edg{\sum_{k=0}^{N-1}\abs{\bbE_{\pi_k}Y_{\pi_{k+1}}- Y_{\pi_k}}^2}
$$
where the supremum is over all stopping time partitions $\Pi:0=\pi_0\leq\pi_1\leq\cdots\leq \pi_N=T$. Moreover, if $A$ is monotone, it can be further simplified to $\bbE\sup_{t\in[0,T]}|Y_t|^2<\infty$.
\end{Remark}
To prove Theorem \ref{main theorem}, we use the following Lemma from \cite{nualart2006malliavin}.
\begin{lemma}[Lemma 1.2.3, \cite{nualart2006malliavin}.]
\label{main lemma}
    Let $(F^n)_{n\geq 1}$ be a sequence of random variables in $\bbD^{1,2}$ that converges to $F$ in $L^2(\Omega)$ and
    \begin{align*}
        \sup_n \bbE\int_0^T|\scD_s F^n|^2ds<\infty.
    \end{align*}
Then $F$ belongs to $\bbD^{1,2}$ and the sequence of derivatives $(\scD F^n)_n$ converges to $\scD F$ in the weak topology of $L^2(\Omega;L^2[0,T])$.
\end{lemma}
\begin{proof}
    Based on the definition of $\hat{V}^L_\mu(loc)$ and Theorem \ref{decomposition theorem}, we know that $f\in \hat{V}^L_\mu(loc)$ and there exists a sequence of $C^{1,2}$ functions $(f^n)_n$ approaches to it. Let $F^n:=f^n(t,X_t)$ where 
    \begin{align*}
    {f}^n(s,x)=n\int_s^{s+1/n}\int_{\bbR^m}f(u,y)P(s,x,u,dy)du,
    \end{align*} 
    as in the proof of Theorem \ref{decomposition theorem}. Then, by the similar argument found in \cite{ladyzenskaya1967linear}, we have
    $\bbE\int_0^T|{f}^n(s,X_s)-f(s,X_s)|^2ds\xrightarrow{n\to\infty}0$, and therefore, 
    $\bbE|F^n-F|^2\xrightarrow{n\to\infty}0$ for almost every $t$, by taking a subsequence if necessary.
    From the chain rule of Malliavin calculus, we can deduce that for each $t$, $F^n$ is Malliavin differentiable and
    \[
    \scD F^n=f^n_x(t,X_t)\scD X_t.
    \]
  Let us fix $T$. Since we can set $\tau_k=k$ in the Proof Step 2 of Theorem \ref{decomposition theorem}, The \eqref{tilde An convergence} and uniform ellipticity of $\sigma(t,x)$  implies $\bbE\int_0^T|f^n_x(s,X_s)|^2ds<\infty$ and $f_x^n(t,X_t)\to f_x(t,X_t)$ in $L^2(\Omega\times[0,T])$. As a result, by taking another subsequence $(f^n)$, we have
    \begin{align*}
        \sup_n \bbE\int_0^t|\scD_s F^n|^2ds&=\sup_n \bbE\int_0^t|f_x^n(t,X_t)\scD_s X_t|^2ds=\esssup_\Omega\int_0^t|\scD_s X_t|^2ds\cdot \sup_n \bbE|f_x^n(t,X_t)|^2<\infty
    \end{align*}
    for almost every $t$.
\end{proof}
\begin{remark}
    If we consider $|b(t,x)|\leq C(1+|x|)$ for some constant $C$, and the SDE is of the form
    \begin{align}
    \label{unbounded drift sde}
    dX_t=b(t,X_t)dt+dW_t;\quad X_0=x\in\bbR^m,
     \end{align}
then the same result holds for Theorem \ref{main theorem}. The reason is we could apply Girsanov theorem to remove the drift and have results in Theorem \ref{decomposition theorem} and semimartingale property is stable under equivalent measure change. Moreover, the constant diffusion term guarantees the Malliavin differentiability of solution $X$.
\end{remark}
\section{Functions transforming a continuous Markov semimartingale into a semimartingale}
In this section we consider a continuous Markov semimartingale $\hat{X}$ defined on the probability space $(\Omega,\scF,\bbP)$, and the filtration $\brak{\scF^{\hat{X}}_t}_t$ is the augmented (strong Markov) filtration generated by $\hat{X}$.

The main result of this section is based on the following observation.
\begin{theorem}
    Let $\hat{X}$ be an $m$-dimensional continuous Markov semimartingale starting at $0$, then for $\hat{X}$, we have 
    \begin{align*}
        \hat{X}_t=A_t+M_t,
    \end{align*}
where $A$ is the finite variation part and $M$ is the local martingale with quadratic variation $N$. For such a process, there exists a system of local characteristics $(V,b,c)$ with $V$ is a strictly increasing continuous process, and $b,c$ are Borel measurable functions such that
\begin{align*}
    A_t=\int_0^tb(\hat{X}_s)dV_s,\quad\text{and}\quad N_t=\int_0^tc^2(\hat{X}_s)dV_s.
\end{align*}
Moreover, under some filtration $(\scG_t)_t$, we have an adapted stochastic process $X$ as a time-changed $\hat{X}$ such that
\begin{align*}
    X_t=\int_0^tb(X_s)ds+\int_0^tc(X_s)dW_s
\end{align*}
is an It\^{o} process. 
\end{theorem}
\begin{proof}
Firstly, one should note that $\hat{X}$ is an additive process since it starts at $0$. Then we define the strictly increasing continuous process as the following
\begin{align*}
    V_t=t+\int_0^t|dA_s|+N_t,
\end{align*}
where $\int_0^t|dA_s|$ stands for the total variation.
Therefore, we have $dA_t\ll dV_t$ and $dN_t\ll dV_t$. Then from section $3f)$, \cite{cinlar1980semimartingales}, we know that there exist Borel measurable functions $b,c$ such that
\begin{align*}
    A_t=\int_0^tb(\hat{X}_s)dV_s,\quad\text{and}\quad N_t=\int_0^tc^2(\hat{X}_s)dV_s.
\end{align*}
Now we can define the random time change 
\begin{align*}
    \hat{V}_u:=\inf\crl{t:V_t>u}.
\end{align*}
We have
\begin{align*}
    \hat{X}_{\hat{V}_t}=:X_t=\int_0^tb(X_s)ds+Y_t,\quad\text{and}\quad\langle Y,Y\rangle_t=\int_0^tc^2(X_s)ds.
\end{align*}
The last equality is obtained from the time change
\begin{align*}
    \langle Y,Y\rangle_t=N_{\hat{V}_t}=\int_0^{\hat{V}_t}c^2(\hat{X}_s)dV_s=\int_0^tc^2(X_s)ds.
\end{align*}
We denote the new filtration as $(\scG_t)_t$ with $\scG_t:=\scF^{\hat{X}}_{\hat{V}_t}$. If $c$ is away from $0$, then we can directly apply Theorem 3.13 in \cite{cinlar1981representation};
if $c$ is possibly $0$, from section 4a of \cite{cinlar1981representation}, we can define 
\begin{align*}
    d\Tilde{W}_t=\frac{1}{c(\hat{X}_t)}\indicator{c(\hat{X}_t)\neq 0}dY_t+\indicator{c(\hat{X}_t)=0}d\Bar{W}_t,
\end{align*}
where $\Bar{W}$ is an independent $\scG_t$-Brownian motion, and then apply that theorem. Therefore, in any sense, there is a $\scG_t$-Brownian motion $W$ such that
\begin{align*}
    X_t=\int_0^tb(X_s)ds+\int_0^tc(X_s)dW_s.
\end{align*}

\end{proof}
\begin{remark}
    For a general continuous Markov semimartingale $X$, we could always construct $Y_t=X_t-X_0$ which is a continuous semimartingale additive functional of $X$.
\end{remark}

\begin{Example}
    We consider a continuous Markov semimartingale $\hat{X}_t=c(t)+t+W_t$ where $c$ is Cantor function and $W$ is Brownian motion. Notice that $V_t=c(t)+t$ is strictly increasing continuous deterministic process, and $dt$ is absolutely continuous with respect to $dV_t$. We have
\begin{align*}
   \langle W,W\rangle_t= t=\int_0^t f^2(\hat{X}_s)dV_s,
\end{align*}
with Lebesgue measure $dt$ is concentrated on the absolutely continuous part (We denote the Cantor set $\scC$.)
\[
f(s) = \frac{dt}{dV}(s) =
\begin{cases}
1, & \text{if } s\notin\scC, \\
0, & \text{if } s \in \scC.
\end{cases}
\]

Now we define the random time change 
\begin{align*}
    \hat{V}_u:=\inf\crl{t:V_t>u}.
\end{align*}
We have
\begin{align*}
    \hat{X}_{\hat{V}_t}=:X_t=t+W_{\hat{V}_t}=t+Y_t,\quad\text{and}\quad\langle Y,Y\rangle_t=\int_0^tf^2(X_s)ds.
\end{align*}
We denote the new filtration as $(\scG_t)_t$.
Moreover, since $f$ is possibly $0$, we have 
\begin{align*}
    d\Tilde{W}_t=\frac{1}{f(X)_t}\indicator{f(X_t)\neq 0}dY_t+\indicator{f(X_t)=0}d\Bar{W}_t,
\end{align*}
where $\Bar{W}$ is an independent $\scG_t$-Brownian motion. In this sense, we have, under new filtration $(\scG_t)_t$,
\begin{align*}
    X_t=t+\int_0^tf(X_s)d\Tilde{W}_s
\end{align*}
which is an It\^{o} process.
\end{Example}

Based on this observation, we have the following as a result.
\begin{theorem}
\label{continuous markov semimtg decomposition theorem}
    Let $\hat{X}$ be an $m$-dimensional continuous Markov semimartingale starting at $x\in\bbR^m$ and satisfy SDE
    \begin{align}
        \label{continuous markov semimtg sde}
        d\hat{X}_t=b(V_t,\hat{X}_t)dV_t+\sigma(V_t,\hat{X}_t)dN_t;\quad \hat{X}_0=x,
    \end{align}
    where $N$ is an adapted local martingale with $\langle N,N \rangle_t=V_t$, and $V$ is an adapted strictly increasing continuous finite variation process. Suppose $b,\sigma$ satisfy Assumption \ref{sde assumption}. We further assume SDE \eqref{ito sde} with coefficients $(b,\sigma)$ has a unique strong solution, for instance, $\sigma$ is Lipschitz. Then $(f(V_t,\hat{X}_t))_{t\geq 0}$ is a semimartingale if and only if $f\in \hat{V}^L_\mu(loc)$, and it has decomposition
\begin{align}
\label{semimtg hat V representation}
    f(V_t,\hat{X}_t)=f(0,\hat{X}_0)+\int_0^t f_x(V_s,\hat{X}_s)d\hat{X}_s+ A_t^f
\end{align}
where $A^f$ is uniquely determined by the relation
\begin{align}
    E\int_0^\infty\psi(V_s,\hat{X}_s)dA_s^f=\iint\psi(s,x)\nu^L_f(ds,dx)
\end{align}
for each bounded continuous function $\psi$.
\end{theorem}
To obtain such result, we borrow arguments introduced by \cite{jacod2006calcul}, and further completed by \cite{kobayashi2011stochastic} on time-changed semimartingales. Since both $V$ and $\hat{V}$ have strictly increasing continuous paths, they satisfy the ``double bracket [[\:]]'' relation in \cite{kobayashi2011stochastic} and the connection between $X$ and $\hat{X}$ can be shown by the following  duality theorem of SDEs.
\begin{theorem}[Theorem $4.2$, \cite{kobayashi2011stochastic}.]
\label{duality theorem}
Let $\hat{X}$ be the process introduced in Theorem \ref{continuous markov semimtg decomposition theorem} and $X$ be its corresponding time-chenged It\^{o} process. We have if $X$ satisfies \eqref{ito sde}, then $\hat{X}$ with $\hat{X}_t=X_{V_t}$ satisfies \eqref{continuous markov semimtg sde}; if $\hat{X}$ satisfies \eqref{continuous markov semimtg sde}, then $X$ with $X_t=\hat{X}_{\hat{V}_t}$ satisfies \eqref{ito sde}.
\end{theorem}
Moreover, \cite{kobayashi2011stochastic} showed that, for any $t\geq 0$, and any well-defined stochastic integral driven by a  semimartingale $Z$, with probability one, we can also have
\begin{align*}
    \int_0^{V_t}H_sdZ_s=\int_0^tH_{V_s}dZ_{V_s}\quad\text{and}\quad\int_0^t H_sdZ_{V_s}=\int_0^{V_t}H_{\hat{V}_s}dZ_s.
\end{align*}

Based on Definition \ref{def V}, we consider the sequence $(h_k)_k$ is given by the following.
\begin{remark}
\label{localizing remark}
    We consider a sequence of stopping times $(\hat{u}_k)_k$ with $\hat{u}_k\rightarrow\infty$, and define a localizing function $\hat{h}_k(s,x)=\bbE\edg{\indicator{s<\hat{u}_k}|\hat{X}_s=x}$. Then we can have for each $k>0$, 
    \begin{align*}
\hat{h}_k(V_t,\hat{X}_t)=\bbE\edg{\indicator{V_t<\hat{u}_k}|\hat{X}_t}=\bbE\edg{\indicator{t<u_k}|X_t}=:h_k(t,X_t),
    \end{align*}
where $u_k:=\inf\crl{t\geq 0:V_t>\hat{u}_k}$. For some $0<\lambda<1$, if we define $C_k=\crl{(s,x):h_k(s,x)>\lambda}$, note that $\brak{\hat{h}_k(V_t,\hat{X}_t)}_t$ and $\brak{h_k(t,X_t)}_t$ are continuous submartingales, and we have
\begin{align*}
    \tau_k:=\inf\crl{t:(t,X_t)\notin C_k}=\inf\crl{t:h_k(t,X_t)\leq\lambda}=\inf\crl{t:\hat{h}_k(V_t,\hat{X}_t)\leq\lambda}=\inf\crl{t:(V_t,\hat{X}_t)\notin C_k}=:\hat{\tau}_k.
\end{align*}
\end{remark}
Now we proceed to prove Theorem \ref{continuous markov semimtg decomposition theorem}.
\begin{proof}
Above observation tells us that the time change transforms the theorem to Theorem \ref{decomposition theorem} and vice versa.
\end{proof}
\begin{remark}
    The same result should hold if we consider the $\hat{X}$ satisfies the SDE
    \begin{align*}
d\hat{X}_t=b(V_t,\hat{X}_t)dV_t+dN_t;\quad \hat{X}_0=x,
    \end{align*}
    where $b$ is allowed to be unbounded and $|b(t,x)|\leq C(1+|x|)$. The reason is equation \eqref{unbounded drift sde} has a unique strong solution by \cite{menoukeu2019flows}. 
\end{remark}
\section{Acknowledgement of AI use}
AI has been used to polish English expressions in abstract.
\bibliographystyle{apalike}
\bibliography{ref.bib}

\end{document}